\documentclass[a4paper,12pt,reqno]{amsart}
\usepackage[text={155mm,220mm}]{geometry}
\usepackage{textcomp}
\theoremstyle{plain}
\newtheorem{theorem}{Theorem}
\newtheorem*{theorem*}{Theorem}
\newtheorem{corollary}{Corollary}
\newtheorem*{corollary*}{Corollary}
\newtheorem{lemma}{Lemma}
\newtheorem*{lemma*}{Lemma}

\newtheorem*{proposition*}{Proposition}

\newtheorem*{conjecture*}{Conjecture}
\theoremstyle{definition}

\newtheorem*{definition*}{Definition}
\theoremstyle{remark}
\newtheorem{remark}{Remark}
\newtheorem*{remark*}{Remark}

\begin{document}

\title[Rational numbers in terms of Cantor series]{Rational numbers represented by sign-variable Cantor series (Rational numbers defined in terms of certain generalized series)}
\author{Symon Serbenyuk}
\address{
  45~Shchukina St. \\
  Vinnytsia \\
  21012 \\
  Ukraine}
\email{simon6@ukr.net}



\begin{abstract} The present article is devoted to representations of rational numbers in terms sign-variable Cantor expansions. The main attention is given to   one of the  discussions  given by J.~Galambos in~\cite{Galambos1976}.
\end{abstract}

\maketitle

\maketitle



\section{Introduction}
In 1869, in the paper \cite{Cantor1}, Georg Cantor   introduced   series of the form
\begin{equation}
\label{eq:  series 1}
\frac{\varepsilon_1}{q_1}+\frac{\varepsilon_2}{q_1q_2}+\dots +\frac{\varepsilon_k}{q_1q_2\dots q_k}+\dots .
\end{equation}
Here $Q\equiv (q_k)$ is a fixed sequence of positive integers, $q_k>1$,  and $(\Theta_k)$ is a sequence of the sets $\Theta_k\equiv\{0,1,\dots ,q_k-1\}$,  as well as $\varepsilon_k\in\Theta_k$.

Note that the last-mentioned  expansion under the condition $q_k=const=q$ for all positive integers $k$, where $1<q\in\mathbb N$ ($\mathbb N$  is the set of all positive integers), is the q-ary expansions of real numbers from~$[0,1]$, i.e., 
$$
\frac{\varepsilon_1}{q}+\frac{\varepsilon_2}{q^2}+\dots+\frac{\varepsilon_k}{q^k}+\dots,
$$
where $\varepsilon_k\in\{0,1, \dots , q-1\}$. In this case, a number is a rational number if and only if a sequence $(\varepsilon_k)$ is periodic.

Series of form \eqref{eq:  series 1} are called \emph{Cantor series}. By $\Delta^Q _{\varepsilon_1\varepsilon_2...\varepsilon_k...}$  denote any number $x\in [0,1]$  having expansion \eqref{eq:  series 1}. This notation is called \emph{the representation of  $x\in [0,1]$ by Cantor series \eqref{eq:  series 1}}. This encoding of real numbers is an example of a polybasic numeral system with zero redundancy and has a finite alphabet when $(q_k)$ is bounded.

Cantor series expansions have been intensively studied from different points of view during the last century (for example, see a brief description in \cite{Serbenyuk17}). For example, the following researchers investigated problems related with Cantor series: P. Erd\"os, J. Galambos, G. Iommi, P. Kirschenhofer,  T. Komatsu, V. Laohakosol, B. Li, M. Pa\v{s}t\'eka, S. Prugsapitak, J. Rattanamoong,  A. R\'enyi, B. Skorulski, R. F. Tichy, P. Tur\'an, Yi Wang, M. S. Waterman, H. Wegmann, Liu Wen, Zhixiong Wen, Lifeng Xi, and other mathematicians. However, many problems related to these series are not solved completely. One can note that  criteria of representation of rational numbers, modeling functions with a complicated local structure are still open problems.

The problem of the rationality/irrationality of numbers defined in terms of generalizations of the q-ary numeral system is difficult. A version of this problem for expansions of form \eqref{eq:  series 1} was introduced in the paper \cite{Cantor1} in 1869 and has been studied by a number of researchers. For example, G. Cantor, P. A. Diananda, A. Oppenheim, P. Erd\"os, J. Han\v{c}l, E.G. Straus, P. Rucki, P. Kuhapatanakul, V. Laohakosol, and other scientists studied this problem (see references in \cite{Serbenyuk17, Serbenyuk20}). However, known results include different conditions satisfied by $(\varepsilon_k)$  and/or $(q_k)$.  For example, in the paper \cite{Cantor1}, necessary and sufficient conditions for a rational number to be representable by series \eqref{eq:  series 1} are formulated for the case when $(q_k)$  is a periodic sequence. A little is known about necessary and sufficient conditions of the rationality (irrationality) for the case of an arbitrary sequence $(q_k)$. One can note  only some  resuts from the papers \cite{{Diananda_Oppenheim1955}, {Hancl_Tijdeman2004}, Ser2017, S13, Serbenyuk2017,  Serbenyuk20, Tijdeman_Pingzhi2002}. 

In~\cite{Galambos1976}, the problem on representations of rational numbers by Cantor series \eqref{eq:  series 1} is called  \emph{the fourth open problem}. In the last-mentioned monograph,   Prof.~J\'anos~Galambos noted the following:

``Problem Four. Give a criterion of rationality for numbers given by a Cantor series. What one should seek here is a directly applicable criterion. A general sufficient condition for rationality would also be of interest, in which the quoted theorems of Diananda and Oppenheim (including the abstract criterion by condensations) can be guides or useful tools.

If in a Cantor series, negative and positive terms are permitted, somewhat less is known about rationality or irrationality of the resulting sum. G. Lord (personal communication) tells me that the condensation method can be extended to this case as well, but still, the results are less complete than in the case of ordinary Cantor series."(\cite[p. 134]{Galambos1976}).
  
The present paper is devoted to representations of rational numbers by positive and sign-variable Cantor series. For postive Cantor series, author's results are recalled. For sign-variable Cantor series, corresponding results are given. However, really,  the results are less complete than in the case of positive Cantor series. 

Let us consider sign-variable Cantor series expansions. Let $\mathbb N_B$ be a fixed subset of positive integers, 
$$
a_n=\begin{cases}
-1&\text{if $n\in \mathbb N_B$}\\
1&\text{if  $n\notin \mathbb N_B$,}
\end{cases}
$$
and $Q\equiv (q_n)$ be a fixed sequence of positive integers such that $q_n>1$ for all $n\in\mathbb N$.   Then we get the following representation of real numbers
\begin{equation}
\label{eq: sign-variable series}
x=\Delta^{(\pm Q, \mathbb N_B)} _{\varepsilon_1\varepsilon_2...\varepsilon_n...}\equiv \sum^{\infty} _{n=1}{\frac{a_n\varepsilon_n}{q_1q_2\dots q_n}}, 
\end{equation}
where $\varepsilon_k\in\{0,1,\dots , q_n-1\}$. 

The last representation is called \emph{ the representation of a number $x$ by a sign-variable Cantor series} or \emph{the quasi-nega-Q-representation}.

One can note that certain numbers from $[0,1]$ have two different representations by Cantor series \eqref{eq:  series 1}, i.e., 
$$
\Delta^Q _{\varepsilon_1\varepsilon_2\ldots\varepsilon_{m-1}\varepsilon_m000\ldots}=\Delta^Q _{\varepsilon_1\varepsilon_2\ldots\varepsilon_{m-1}[\varepsilon_m-1][q_{m+1}-1][q_{m+2}-1]\ldots}=\sum^{m} _{i=1}{\frac{\varepsilon_i}{q_1q_2\dots q_i}}.
$$
Such numbers are called \emph{$Q$-rational}. The other numbers in $[0,1]$ are called \emph{$Q$-irrational}.

\begin{theorem}[\cite{Ser2017, {S13}}]
\label{theorem2}
A rational number $\frac{p}{r}\in(0,1)$ has two different representations if and only if  there exists a number $n_0$ such that
 $$
q_1q_2\dots q_{n_0} \equiv 0\pmod{r}.
$$
\end{theorem}

Let us remark that the necessity of this theorem is given in~\cite{Cantor1} with the other formulation and with a more complicated proof. The interest to the last theorem is in the following corollary.

\begin{corollary}[\cite{Ser2017, {S13}}]
There exist certain sequences $(q_k)$ such that all rational numbers represented in terms of  corresponding Cantor series have  finite expansions. 
\end{corollary}


\begin{remark}
In terms of alternating or sign-variable Cantor series, the last theorem is the  necessary and sufficient condition
for a rational number to have a finite expansion  by an alternating or sign-variable Cantor series. 
\end{remark}

Let us recall the notion of  \emph{the shift operator $\sigma$ of expansion \eqref{eq:   series 1}} which is defined  by the rule
$$
\sigma(x)=\sigma\left(\Delta^Q _{\varepsilon_1\varepsilon_2\ldots\varepsilon_k\ldots}\right)=\sum^{\infty} _{k=2}{\frac{\varepsilon_k}{q_2q_3\dots q_k}}=q_1\Delta^{Q} _{0\varepsilon_2\ldots\varepsilon_k\ldots}.
$$

It is easy to see that 
\begin{equation*}
\begin{split}
\sigma^n(x) &=\sigma^n\left(\Delta^Q _{\varepsilon_1\varepsilon_2\ldots\varepsilon_k\ldots}\right)\\
& =\sum^{\infty} _{k=n+1}{\frac{\varepsilon_k}{q_{n+1}q_{n+2}\dots q_k}}=q_1\dots q_n\Delta^{Q} _{\underbrace{0\ldots 0}_{n}\varepsilon_{n+1}\varepsilon_{n+2}\ldots}.
\end{split}
\end{equation*}

Therefore, 
\begin{equation*}
x=\sum^{n} _{i=1}{\frac{\varepsilon_i}{q_1q_2\dots q_i}}+\frac{1}{q_1q_2\dots q_n}\sigma^n(x).
\end{equation*}

The following two equivalent theorems are general necessary and sufficient conditions for a rational number to be representable by positive Cantor series.

\begin{theorem}[\cite{Ser2017, {S13}, Serbenyuk17}]
\label{th: the main theorem}
A number $x$ represented by series \eqref{eq:   series 1} is  rational if and only if  there exist numbers $n\in\mathbb Z_0$ and $m\in\mathbb N$ such that $\sigma^n(x)=\sigma^{n+m}(x)$.
\end{theorem}

The last theorem is true for the case of alternating Cantor series (\cite{Serbenyuk2017}). Here $\mathbb Z_0=\mathbb N\cup \{0\}$.

\begin{theorem}[\cite{Ser2017, {S13}, Serbenyuk17}]
\label{th: the main theorem 2}
A number $x=\Delta^Q _{\varepsilon_1\varepsilon_2\ldots \varepsilon_k\ldots }$  is  rational if and only if there exist numbers $n\in\mathbb Z_0$ and $m\in\mathbb N$ such that
$$
\Delta^{Q} _{\underbrace{0\ldots 0}_{n}\varepsilon_{n+1}\varepsilon_{n+2}\ldots }=q_{n+1}\dots q_{n+m}\Delta^{Q} _{\underbrace{0\ldots 0}_{n+m}\varepsilon_{n+m+1}\varepsilon_{n+m+2}\ldots }.
$$
\end{theorem}

In this paper, it will be proven that these two theorems are true for the case of  sign-variable Cantor series, where $\sigma$ is the shift operator of a sign-variable Cantor expansion. 

Let us recall several auxiliary statements which are true for positive Cantor series but do not hold for sign-variable Cantor series.

\begin{lemma}[\cite{Ser2017, {S13}}]
Let $n_0$ be a fixed positive integer number. Then the condition $\sigma^n(x)=const$ holds for all $n\ge n_0$ if and only if  $\frac{\varepsilon_n}{q_n-1}=const$ for all $n>n_0$.
\end{lemma}

\begin{lemma}[\cite{Ser2017, {S13}}]
Suppose we have $q=\min_n{q_n}$ and  fixed $\varepsilon\in\{0,1,\dots ,q-1\}$. Then the condition $\sigma^n(x)=x=\frac{\varepsilon}{q-1}$ holds if and only if the condition $\frac{q_n-1}{q-1}\varepsilon=\varepsilon_n\in\mathbb Z_0$ holds for all $n\in\mathbb N$.
\end{lemma}

\begin{corollary}[\cite{Ser2017, {S13}}]
Let $n_0$ be a fixed positive integer number, $q_0=\min_{n>n_0}{q_n}$, and $\varepsilon_0$ be a numerator of the fraction 
$\frac{\varepsilon_{n_0+k}}{q_1q_2...q_{n_0}q_{n_0+1}...q_{n_0+k}}$ in expansion \eqref{eq:   series 1} of $x$ providing that $q_{n_0+k}=q_0$. Then $\sigma^n(x)=const$ for all $n\ge n_0$ if and only if the condition $\frac{q_n-1}{q_0-1}\varepsilon_0=\varepsilon_n\in\mathbb Z_0$ holds for any $n>n_0$.
\end{corollary}

Let us remark that the paper \cite{S13} is the translated into English   version of \cite{Ser2017}. 

The following theorem can be generalized from the case of positive Cantor series to the case of sign-variables series. However, since the theorem is complex (has complicated formulation, difficult applicable), this statement has an auxiliary character. 
 
\begin{theorem}[\cite{Ser2017, {S13}}]
A number $x$ represented by expansion \eqref{eq:   series 1} is rational if and only if there exists a subsequence $(n_k)$ of positive integers  such that for all $k=1,2,\dots ,$ the following conditions are true:
\begin{itemize}
\item
$$
\frac{\lambda_k}{\mu_k}=
\frac{\varepsilon_{n_k+1}q_{n_k+2}\dots q_{n_{k+1}}+\varepsilon_{n_k+2}q_{n_k+3}\dots q_{n_{k+1}}+\dots +\varepsilon_{n_{k+1}-1}q_{n_{k+1}}+\varepsilon_{n_{k+1}}}{q_{n_k+1}q_{n_k+2}\dots q_{n_{k+1}}-1}=const;
$$
\item $\lambda_k=\frac{\mu_k}{\mu}\lambda$, where $\mu=\min_{k\in\mathbb N}{\mu_k}$ and $\lambda$ is a number in the numerator of the fraction whose denominator equals $(\mu_1+1)(\mu_2+~1)\dots (\mu+~1)$  from sum \eqref{eq: Cantor series 6}.
\end{itemize}
\end{theorem}
Here 
$$
x=\sum^{\infty} _{k=1}{\frac{\varepsilon_k}{q_1q_2\dots q_k}}=\sum^{n_1} _{j=1}{\frac{\varepsilon_j}{q_1q_2\dots q_j}}+\frac{1}{q_1q_2\dots q_{n_1}}x^{'},
$$
$$
x^{'}=\sum^{\infty} _{k=1}{\frac{\varepsilon_{n_k+1}q_{n_k+2}q_{n_k+3}\dots q_{n_{k+1}}+\varepsilon_{n_k+2}q_{n_k+3}\dots q_{n_{k+1}}+\varepsilon_{n_{k+1}-1}q_{n_{k+1}}+\varepsilon_{n_{k+1}}}{(q_{n_1+1}\dots q_{n_2})(q_{n_2+1}\dots q_{n_3})\dots (q_{n_k+1}\dots q_{n_{k+1}})}}
$$
\begin{equation}
\label{eq: Cantor series 6}
=\sum^{\infty} _{k=1}{\frac{\lambda_k}{(\mu_1+1)\dots (\mu_k+1)}}.
\end{equation}
 In the case of series~\eqref{eq: Cantor series 6}, the condition $\sigma^k(x^{'})=const$ holds for all $k=~0,1,\dots$.

Finally, let us consider the following necessary and sufficient condition which is useful for modeling rational numbers by positive Cantor series. 

\begin{theorem}[\cite{Serbenyuk20}]
\label{th: modeling}
A number $x=\Delta^Q _{\varepsilon_1\varepsilon_2...\varepsilon_n...} \in (0,1)$ represented by series \eqref{eq: series 1}  is a rational number $\frac{p}{r}$, where $p,r\in\mathbb N, (p,r)=1$, and $p<r$, if and only if the condition 
$$
\varepsilon_n=\left[\frac{q_n(\Delta_{n-1}-r\varepsilon_{n-1})}{r}\right]
$$
 holds for all $1<n\in\mathbb N$, where $\Delta_1=pq_1$, $\varepsilon_1=\left[\frac{\Delta_1}{r}\right]$, and 
 $[a]$ is the integer part of $a$.
\end{theorem}


\section{Sign-variable Cantor series, cylinders,  and the shift-operator}

Using the analogy used in proofs in~\cite{S.Serbenyuk, Serbenyuk2017}, one can prove that each number $x$ from a certain closed interval  $ [a^{'} _{\pm Q}, a^{''} _{\pm Q}]$ can be represented by  series~\eqref{eq: sign-variable series}.
Since
$$
\sum^{\infty} _{k=1}{\frac{q_k-1}{q_1q_2\cdots q_k}}=1,
$$
we have
$$
a^{'} _{\pm Q}=-\sum_{1 \le k \in N_B} {\frac{q_k-1}{q_1q_2\cdots q_k}}
$$
and
$$
a^{''} _{\pm Q}=\sum_{1\le k \notin N_B} {\frac{q_k-1}{q_1q_2\cdots q_k}}=1- \sum_{1 \le k \in N_B} {\frac{q_k-1}{q_1q_2\cdots q_k}}
$$

Certain numbers from  $ [a^{'} _{\pm Q}, a^{''} _{\pm Q}]$ have two different  quasi-nega-Q-representations and  are called \emph{quasi-nega-Q-rational}.  There are numbers of the following forms: 
\begin{enumerate}
\item if $k,k+1 \in N_B$, then
$$
\Delta^{(\pm Q, N_B)} _{\varepsilon_1\varepsilon_2...\varepsilon_{k-1}[q_k-1-\varepsilon_k][q_{k+1}-1]\beta_{k+2}\beta_{k+3}...}=\Delta^{(\pm Q, N_B)} _{\varepsilon_1\varepsilon_2...\varepsilon_{k-1}[q_k-\varepsilon_k]0\gamma_{k+2}\gamma_{k+3}...};
$$
\item if $k \in N_B$, $k+1 \notin N_B$, then
$$
\Delta^{(\pm Q, N_B)} _{\varepsilon_1\varepsilon_2...\varepsilon_{k-1}[q_k-1-\varepsilon_k]0\beta_{k+2}\beta_{k+3}...}=\Delta^{(\pm Q, N_B)} _{\varepsilon_1\varepsilon_2...\varepsilon_{k-1}[q_k-\varepsilon_k][q_{k+1}-1]\gamma_{k+2}\gamma_{k+3}...};
$$
\item if $k \notin N_B$, $k+1 \in N_B$, then
$$
\Delta^{(\pm Q, N_B)} _{\varepsilon_1\varepsilon_2...\varepsilon_{k-1}\varepsilon_k[q_{k+1}-1]\beta_{k+2}\beta_{k+3}...}=\Delta^{(\pm Q, N_B)} _{\varepsilon_1\varepsilon_2...\varepsilon_{k-1}[\varepsilon_k-1]0\gamma_{k+2}\gamma_{k+3}...};
$$
\item if $k \notin N_B$, $k+1 \notin N_B$, then
$$
\Delta^{(\pm Q, N_B)} _{\varepsilon_1\varepsilon_2...\varepsilon_{k-1}\varepsilon_k0\beta_{k+2}\beta_{k+3}...}=\Delta^{(\pm Q, N_B)} _{\varepsilon_1\varepsilon_2...\varepsilon_{k-1}[\varepsilon_k-1][q_{k+1}-1]\gamma_{k+2}\gamma_{k+3}...}.
$$
\end{enumerate}
Here  $\varepsilon_k\ne 0$, 
\begin{equation*}
\label{eq: znakozminnyi-s-rozklad 2}
\beta_k=\begin{cases}
0&\text{whenever $k \notin N_B$}\\
q_k-1&\text{whenever $k \in N_B$,}
\end{cases}
\end{equation*}
and 
\begin{equation*}
\label{eq: znakozminnyi-s-rozklad 3}
\gamma_k=\begin{cases}
q_k-1&\text{whenever $k \notin N_B$}\\
0&\text{whenever $k \in N_B$.}
\end{cases}
\end{equation*}

The other numbers in $ [a^{'} _{\pm Q}, a^{''} _{\pm Q}]$  are called \emph{quasi-nega-Q-irrational} and have the unique quasi-nega-Q-representation.

Let $c_1,c_2, \dots , c_m$ be an ordered tuple of integers such that $c_i\in \{0,1, \dots , q_i-1\}$ for all $i=\overline{1,m}$. Then
\emph{a cylinder $\Lambda^{(\pm Q, N_B)} _{c_1c_2...c_m}$ of rank $m$ with base $c_1c_2\ldots c_m$} is a set of the form
$$
\Lambda^{(\pm Q, N_B)} _{c_1c_2...c_m}\equiv\{x: x=\Delta^{(\pm Q, N_B)} _{c_1c_2...c_m\varepsilon_{m+1}\varepsilon_{m+2}\ldots\varepsilon_{m+k}\ldots}\}.
$$

\begin{lemma}
A  cylinder $\Lambda^{(\pm Q, N_B)} _{c_1c_2...c_m}$  is a  closed interval, i.e.,
$$
\Lambda^{(\pm Q, N_B)} _{c_1c_2...c_m}=\left[\sum^{m} _{i=1}{\frac{a_ic_i}{q_1q_2\cdots q_i}}-\sum_{m< k \in N_B} {\frac{q_k-1}{q_1q_2\cdots q_k}}, \sum^{m} _{i=1}{\frac{a_ic_i}{q_1q_2\cdots q_i}}+\sum_{m< k \notin N_B} {\frac{q_k-1}{q_1q_2\cdots q_k}}\right].
$$
\end{lemma}
\begin{proof}
Let    $x \in \Lambda^{(\pm Q, N_B)} _{c_1c_2\ldots c_m}$, i.e., 
$$
x=\sum^{m} _{i=1} {\frac{a_ic_i}{q_1q_2\cdots q_i}}+\sum^{\infty} _{j=m+1} {\frac{a_j\varepsilon_j}{q_1q_2\cdots q_j}}, 
$$
where $ \varepsilon_j \in \{0,1,\dots ,q_j-1\}$; then
$$
x^{'}=\sum^{m} _{i=1}{\frac{a_ic_i}{q_1q_2\cdots q_i}}-\sum_{m< k \in N_B} {\frac{q_k-1}{q_1q_2\cdots q_k}} \le x\le \sum^{m} _{i=1}{\frac{a_ic_i}{q_1q_2\cdots q_i}}+\sum_{m< k \notin N_B} {\frac{q_k-1}{q_1q_2\cdots q_k}}=x^{''}.
$$
Hence  $x \in [x^{'},x^{''}]$ and $\Lambda^{(\pm Q, N_B)} _{c_1c_2\ldots c_m}\subseteq [x^{'},x^{''}]$.

Since the equalities 
$$
-\sum_{m< k \in N_B} {\frac{q_k-1}{q_1q_2\cdots q_k}}=\frac{1}{q_1q_2\cdots q_m}\inf\sum^{\infty} _{j=m+1}{\frac{a_j \varepsilon_j}{q_{m+1}q_{m+2}\cdots q_j}}
$$
and 
$$
\sum_{m< k \notin N_B} {\frac{q_k-1}{q_1q_2\cdots q_k}}=\frac{1}{q_1q_2\cdots q_m}\sup\sum^{\infty} _{j=m+1}{\frac{a_j \varepsilon_j}{q_{m+1}q_{m+2}\cdots q_j}}
$$
hold, we have 
 $x \in \Lambda^{(\pm Q, N_B)} _{c_1c_2\ldots c_m}$ and $x^{'}, x^{''} \in \Lambda^{(\pm Q, N_B)} _{c_1c_2\ldots c_m}$.
\end{proof}

Let us consider the useful notion of the shift operator of expansion \eqref{eq: sign-variable series}.  Define \emph{the shift operator $\sigma$ of expansion \eqref{eq: sign-variable series}} by the rule
$$
\sigma(x)=\sigma\left(\Delta^{(\pm Q, N_B)} _{\varepsilon_1\varepsilon_2\ldots\varepsilon_k\ldots}\right)=\sum^{\infty} _{k=2}{\frac{a_k\varepsilon_k}{q_2q_3\dots q_k}}=q_1\Delta^{(\pm Q, N_B)} _{0\varepsilon_2\ldots\varepsilon_k\ldots}.
$$
It is easy to see that 
\begin{equation}
\label{eq:  series 2}
\begin{split}
\sigma^n(x) &=\sigma^n\left(\Delta^{(\pm Q, N_B)} _{\varepsilon_1\varepsilon_2\ldots\varepsilon_k\ldots}\right)\\
& =\sum^{\infty} _{k=n+1}{\frac{a_k\varepsilon_k}{q_{n+1}q_{n+2}\dots q_k}}=q_1\dots q_n\Delta^{(\pm Q, N_B)} _{\underbrace{0\ldots 0}_{n}\varepsilon_{n+1}\varepsilon_{n+2}\ldots}.
\end{split}
\end{equation}
Therefore, 
\begin{equation}
\label{eq:  sign-variable series 3}
x=\sum^{n} _{i=1}{\frac{a_i\varepsilon_i}{q_1q_2\dots q_i}}+\frac{1}{q_1q_2\dots q_n}\sigma^n(x).
\end{equation}

\section{Rational numbers defined in terms of sign-variable Cantor series}

\begin{theorem}
\label{th: the main theorem}
A number $x$ represented by series \eqref{eq: sign-variable series} is  rational for the case of any $N_B \subseteq~\mathbb N$ if and only if  there exist numbers $n\in\mathbb Z_0$ and $m\in\mathbb N$ such that $\sigma^n(x)=\sigma^{n+m}(x)$.
\end{theorem}
\begin{proof} Let us prove that  \emph{the necessity} is true.  Suppose we have a rational number $x=\frac{u}{v}$, where $|u|<v$, $v\in\mathbb N, u\in \mathbb Z$, and $(|u|,v)=1$. Consider the sequence $(\sigma^n(x))$ generated by the shift operator of expansion \eqref{eq: sign-variable series} of the number $x$. That is,
\begin{gather*}
\sigma^0(x)=x,\\
\sigma(x)=q_1x-a_1\varepsilon_1,\\
\sigma^2(x)=q_2\sigma(x)-a_2\varepsilon_2=q_1q_2x-a_1q_2\varepsilon_1-a_2\varepsilon_2,\\
\dots \dots \dots \dots \dots \dots \dots \\
\sigma^n(x)=q_n\sigma^{n-1}(x)-a_n\varepsilon_{n}=x\prod^{n} _{i=1}{q_i}-\left(\sum^{n-1} _{j=1}{a_j\varepsilon_jq_{j+1}q_{j+2}\dots q_n}\right)-a_n\varepsilon_n,\\
\dots \dots \dots \dots \dots \dots \dots
\end{gather*}
Hence using equality \eqref{eq:  sign-variable series 3}, we obtain
\begin{equation*}
\label{eq: Cantor series4}
\sigma^n(x)=\frac{uq_1q_2\dots q_n-v(a_1\varepsilon_1q_2\dots q_n+\dots +a_{n-1}\varepsilon_{n-1}q_n+a_n\varepsilon_n)}{v}=\frac{u_n}{v}.
\end{equation*}

Let us remark that $|\sigma^n(x)|=\left|\sigma^n\left(\Delta^{(\pm Q, N_B)} _{\varepsilon_1\varepsilon_2...\varepsilon_k...}\right)\right|\le 1$ for any set $N_B$.

 Since $v=const$ and the condition $\frac{|u_n|}{|v|}\le 1$ holds as $n\to\infty$, we have 
$$
u_n\in\{-v, -(v-1), \dots , -1, 0,1,\dots ,v-1, v\}.
$$
 Thus there exist numbers  $n, m\in\mathbb N$ such that $u_n=u_{n+m}$. In addition, there exists a sequence $(n_k)$   of positive integers such that $u_{n_k}=const$. 

Let us prove \emph{the sufficiency.}  Suppose there exist $n\in \mathbb Z_0=\mathbb N \cup \{0\}$ and $m\in\mathbb N$ such that $\sigma^n(x)=\sigma^{n+m}(x)$. Then
$$
\sigma^n(x) =\sigma^n\left(\Delta^{(\pm Q, N_B)} _{\varepsilon_1\varepsilon_2\ldots\varepsilon_k\ldots}\right)=q_1\dots q_n\Delta^{(\pm Q, N_B)} _{\underbrace{0\ldots 0}_{n+m}\varepsilon_{n+m+1}\varepsilon_{n+m+2}\ldots},
$$
$$
\sigma^{n+m}(x) =\sigma^{n+m}\left(\Delta^{(\pm Q, N_B)} _{\varepsilon_1\varepsilon_2\ldots\varepsilon_k\ldots}\right)=q_1\dots q_n\dots q_{n+m}\Delta^{(\pm Q, N_B)} _{\underbrace{0\ldots 0}_{n}\varepsilon_{n+1}\varepsilon_{n+2}\ldots},
$$
and 
$$
\sigma^n(x) =\sigma^{n+m}(x) =\frac{q_1q_2\dots q_nq_{n+1}\dots q_{n+m}}{q_{n+1}\dots q_{n+m}-1}\Delta^{(\pm Q, N_B)} _{\underbrace{0\ldots 0}_n\varepsilon_{n+1}\varepsilon_{n+2}\ldots \varepsilon_{n+m}000\ldots }.
$$

 So, from equality \eqref{eq:  series 2} it follows that 
$$
x=\Delta^{(\pm Q, N_B)} _{\varepsilon_1\varepsilon_2\ldots \varepsilon_n000\ldots }+\frac{q_{n+1}\dots q_{n+m}}{q_{n+1}\dots q_{n+m}-1}\Delta^{(\pm Q, N_B)} _{\underbrace{0\ldots 0}_n\varepsilon_{n+1}\varepsilon_{n+2}\ldots \varepsilon_{n+m}000\ldots }.
$$
That is, $x$ is a rational number. 
\end{proof}

It is easy to see that the last and following theorems are equivalent. 

\begin{theorem}
\label{th: the main theorem 2}
A number $x=\Delta^{(\pm Q, N_B)} _{\varepsilon_1\varepsilon_2\ldots \varepsilon_k\ldots }$  is  rational if and only if there exist numbers $n\in\mathbb Z_0$ and $m\in\mathbb N$ such that
$$
\Delta^{(\pm Q, N_B)} _{\underbrace{0\ldots 0}_{n}\varepsilon_{n+1}\varepsilon_{n+2}\ldots }=q_{n+1}\dots q_{n+m}\Delta^{(\pm Q, N_B)} _{\underbrace{0\ldots 0}_{n+m}\varepsilon_{n+m+1}\varepsilon_{n+m+2}\ldots }.
$$
\end{theorem}

In the case of sign-variable Cantor series, a version of Theorem \ref{th: modeling} has some  differences. 

\begin{theorem}
If $x=\Delta^{(\pm Q, N_B)} _{\varepsilon_1\varepsilon_2...\varepsilon_k...} =\frac{p}{r}$, where $p\in\mathbb Z, r\in\mathbb N, (|p|, r)=1$, and $|p|<r$, then the condition 
$$
\varepsilon_k=\left|\left[\frac{q_k(\Delta^{(g)} _{k-1}-a_{k-1}r\varepsilon_{k-1})}{r}+s_n\right]\right|
$$
 holds for all $1<k\in\mathbb N$. Here $\Delta^{(g)} _1=pq_1$,  $\varepsilon_1=\left|\left[\frac{\Delta^{(g)} _1}{r}+s_1\right]\right|$, and
 $[a]$ is the integer part of $a$. Also, 
$$
s_1=\sum_{1<k \in N_B}{\frac{q_k-1}{q_2q_3\cdots q_k}},
\qquad
s_k=\begin{cases}
q_ks_{k-1}&\text{whenever $k\notin N_B$}\\
q_ks_{k-1}-(q_k-1)&\text{whenever $k\in N_B$.}
\end{cases}
$$
\end{theorem}
\begin{proof}
Let $\frac{p}{r}$ be a fixed number, where $(|p|, r)=1$, $|p|<r$, and $p\in\mathbb Z, r\in\mathbb N$.  
Then
$$
\frac{p}{r}=\sum^{\infty} _{k=1}{\frac{a_k\varepsilon_k}{q_1q_2\cdots q_k}}.
$$
\begin{remark*}
Since $ x\in \Lambda^{(\pm Q, N_B)} _{c_1c_2...c_m}$
but 
$$
\sup\Lambda^{(\pm Q, N_B)} _{c_1c_2...c_m}=\sum^{m} _{i=1}{\frac{a_ic_i}{q_1q_2\cdots q_i}}+\sum_{m<k\notin N_B}{\frac{q_k-1}{q_1q_2 \cdots q_k}}
$$
$$
=\sum^{m} _{i=1}{\frac{a_ic_i}{q_1q_2\cdots q_i}}+\frac{1}{q_1q_2 \cdots q_m}\left(1-\sum_{m<k\in N_B}{\frac{q_k-1}{q_{m+1}q_{m+2} \cdots q_k}}\right)
$$
$$
=\begin{cases}
 \inf\Lambda^{(\pm Q, N_B)} _{c_1c_2...c_{m-1}[c_m-1]}&\text{whenever $m \in N_B$ and $c_m\ne 0$}\\
 \inf\Lambda^{(\pm Q, N_B)} _{c_1c_2...c_{m-1}[c_m+1]} &\text{whenever $m\notin N_B$ and $c_m \ne q_m-1$,}
\end{cases}
$$
we assume that
$$
\inf\Lambda^{(\pm Q, N_B)} _{c_1c_2...c_{m-1}c_m }\le x<\sup\Lambda^{(\pm Q, N_B)} _{c_1c_2...c_{m-1}c_m}.
$$
\end{remark*}
 
It is easy to see that
$$
\frac{p}{r}\in\Lambda^{(\pm Q, N_B)} _{\varepsilon_1}=\left[\frac{a_1\varepsilon_1}{q_1}-\sum_{1<k\in N_B}{\frac{q_k-1}{q_1q_2\cdots q_k}},\frac{a_1\varepsilon_1}{q_1}+\sum_{1<k\notin N_B}{\frac{q_k-1}{q_1q_2\cdots q_k}},\right].
$$
That is,  
$$
\frac{a_1\varepsilon_1}{q_1}-\sum_{1<k\in N_B}{\frac{q_k-1}{q_1q_2\cdots q_k}}\le \frac{p}{r}< \frac{a_1\varepsilon_1}{q_1}+\frac{1}{q_1}-\sum_{1<k\in N_B}{\frac{q_k-1}{q_1q_2\cdots q_k}},
$$
$$
a_1\varepsilon_1\le \frac{pq_1}{r}+\sum_{1<k\in N_B}{\frac{q_k-1}{q_2q_3\cdots q_k}}<a_1\varepsilon_1+1.
$$

If $1\notin N_B$, then 
$$
\varepsilon_1\le \frac{pq_1}{r}+\sum_{1<k\in N_B}{\frac{q_k-1}{q_2q_3\cdots q_k}}<\varepsilon_1+1.
$$

If $1\in N_B$, then 
$$
-\varepsilon_1\le \frac{pq_1}{r}+\sum_{1<k\in N_B}{\frac{q_k-1}{q_2q_3\cdots q_k}}<-(\varepsilon_1-1).
$$
So,
$$
\varepsilon_1=\left|\left[\frac{p}{r}q_1+\sum_{1<k\in N_B}{\frac{q_k-1}{q_2q_3\cdots q_k}}\right]\right|,
$$
where $[x]$ is the integer part of $x$. 

Now we get  $\frac{p}{r}\in \Lambda^{(\pm Q, N_B)} _{\varepsilon_1\varepsilon_2}=$
$$
=\left[\frac{a_1q_2\varepsilon_1+a_2\varepsilon_2}{q_1q_2}-\sum_{2<k\in N_B}{\frac{q_k-1}{q_1q_2\cdots q_k}},\frac{a_1q_2\varepsilon_1+a_2\varepsilon_2}{q_1q_2}+\sum_{2<k\notin N_B}{\frac{q_k-1}{q_1q_2\cdots q_k}}\right].
$$
Whence,
$$
\frac{a_1q_2\varepsilon_1+a_2\varepsilon_2}{q_1q_2}-\sum_{2<k\in N_B}{\frac{q_k-1}{q_1q_2\cdots q_k}}\le \frac{p}{r}<\frac{a_1q_2\varepsilon_1+a_2\varepsilon_2+1}{q_1q_2}-\sum_{2<k \in N_B}{\frac{q_k-1}{q_1q_2\cdots q_k}},
$$
$$
a_2\varepsilon_2\le \frac{pq_1q_2-a_1rq_2\varepsilon_1}{r}+\sum_{2<k \in N_B}{\frac{q_k-1}{q_3q_4\cdots q_k}}<a_2\varepsilon_2+1.
$$
So,
$$
\varepsilon_2=\left|\left[\frac{pq_1q_2-a_1rq_2\varepsilon_1}{r}+\sum_{2<k \in N_B}{\frac{q_k-1}{q_3q_4\cdots q_k}}\right]\right|. 
$$

In the third step, we have $\frac{p}{r}\in \Lambda^{(\pm Q, N_B)} _{\varepsilon_1\varepsilon_2\varepsilon_3}$,
$$
\frac{a_1\varepsilon_1q_2q_3+a_2\varepsilon_2q_3+a_3\varepsilon_3}{q_1q_2q_3}- \sum_{3<k \in N_B}{\frac{q_k-1}{q_1q_2\cdots q_k}} \le x,
$$
and
$$
\frac{ {a_1\varepsilon_1q_2q_3+a_2\varepsilon_2q_3+a_3\varepsilon_3}}{q_1q_2q_3}+\sum_{3<k \notin N_B}{\frac{q_k-1}{q_1q_2\cdots q_k}}>x.
$$
Hence
$$
\varepsilon_3=\left|\left[\frac{pq_1q_2q_3-r(a_1\varepsilon_1q_2q_3+a_2\varepsilon_2q_3)}{r}+\sum_{3<k \in N_B}{\frac{q_k-1}{q_4q_5\cdots q_k}}\right]\right|.
$$

Let $\varsigma^{(g)} _k$ denote the sum $a_1\varepsilon_1q_2q_3\cdots q_k+a_2\varepsilon_2q_3q_4\cdots q_k+\dots+a_{k-1}\varepsilon_{k-1}q_k$ and 
$$
s_k=\sum_{k<n \in N_B}{\frac{q_n-1}{q_{k+1}q_{k+2}\cdots q_n}}=\begin{cases}
q_ks_{k-1}&\text{whenever $k\notin N_B$}\\
q_ks_{k-1}-(q_k-1)&\text{whenever $k\in N_B$,}
\end{cases}
$$
where
$$
s_1=\sum_{1<k \in N_B}{\frac{q_k-1}{q_2q_3\cdots q_k}}.
$$

 Then in the $k$th step, we obtain
$$
\frac{p}{r}\in \Lambda^{(\pm Q, N_B)} _{\varepsilon_1\varepsilon_2...\varepsilon_{k-1}\varepsilon_k}=\left[\sum^{k} _{i=1}{\frac{a_i\varepsilon_i}{q_1q_2\cdots q_i}}-\sum_{k<n \in N_B}{\frac{q_n-1}{q_1q_2\cdots q_n}}
,\sum^{k} _{i=1}{\frac{a_i\varepsilon_i}{q_1q_2\cdots q_i}}+\sum_{k<n \notin N_B}{\frac{q_n-1}{q_1q_2\cdots q_n}}\right],
$$
$$
\frac{\varsigma^{(g)} _k+a_k\varepsilon_k}{q_1q_2\cdots q_k}-\sum_{k<n \in N_B}{\frac{q_n-1}{q_1q_2\cdots q_n}}\le \frac{p}{r}<\frac{\varsigma^{(g)} _k+a_k\varepsilon_n+1}{q_1q_2\cdots q_k}-\sum_{k<n \in N_B}{\frac{q_n-1}{q_1q_2\cdots q_n}},
$$
$$
a_k\varepsilon_k\le\frac{pq_1q_2\cdots q_k-r\varsigma^{(g)} _k}{r}+s_k<a_k\varepsilon_k+1.
$$

So,
$$
\varepsilon_k=\left| \left[ \frac{pq_1q_2\cdots q_k-r\varsigma^{(g)} _k}{r}+s_k \right]\right|
$$

Denoting by $\Delta^{(g)} _k=pq_1q_2\cdots q_k-r\varsigma^{(g)} _k$, we get
$$
\varepsilon_k=\left|\left[\frac{\Delta^{(k)} _k}{r}+s_k\right]\right|.
$$
Also, for $k\ge 2$ the condition $\varsigma^{(g)} _k=\varsigma^{(g)} _{k-1}q_k+a_{k-1}\varepsilon_{k-1}q_k$ holds and
$$
\Delta^{(g)} _k=q_k(\Delta^{(g)} _{k-1}-a_{k-1}r\varepsilon_{k-1}).
$$
 This completes the proof. \end{proof}

Let us remark that in this topic, the notion of the shift operator has an important role. Since
$$
x=\frac{p}{r}=\frac{\varsigma^{(g)} _n+a_n\varepsilon_n}{q_1q_2\cdots q_n}+\frac{\sigma^n(x)}{q_1q_2\cdots q_n},
$$
we have
$$
\sigma^n(x)+a_n\varepsilon_n=\frac{\Delta^{(g)} _n}{r}.
$$
Whence  $\{\frac{\Delta^{(g)} _n}{r}\}$ is equal to $\sigma^n(x)$ whenever $\sigma^n(x)\ge 0$ or equals $1- \sigma^n(x)$ whenever $\sigma^n(x)<0$. Here $\{a\}$ is the fractional part of $a$ (i.e., $a=[a]+\{a\}$).

In the case of positive Cantor series, $\sigma^n(x)=\{\frac{\Delta_n}{r}\}$ and $\varepsilon_n=[\frac{\Delta_n}{r}]$ (see~\cite{Serbenyuk20}).


\begin{thebibliography}{9}

\bibitem{Cantor1} G.~Cantor, Ueber die einfachen Zahlensysteme, 
\emph{Z. Math. Phys. }   {\bf 14} (1869), 121--128.


\bibitem{Diananda_Oppenheim1955} P.~H.~Diananda and A.~Oppenheim, Criteria for irrationality of certain classes of numbers II, 
\emph{Amer. Math. Monthly} \textbf{62} (1955), no.~4, 222--225.
 

\bibitem{Erdos_Straus1974} P.~Erd\"os and Straus~E.~G., On the irrationality of certain series, \emph{Pacific J. Math.}  \textbf{55} (1974),  no.~1, 85--92.

\bibitem{Galambos1976}J. J. Galambos, \emph{Representations of real numbers by infinite series}, Lecture Notes in
Mathematics. 502, Springer, 1976.


\bibitem{Hancl97} J.~Han\v{c}l,  A note to the rationality of infinite series I, 
\emph{Acta Math. Inf. Univ. Ostr.} \textbf{5} (1997), no.~1, 5--11.

\bibitem{Hancl_Rucki2006}J.~Han\v{c}l and P.~Rucki, A note to the transcendence of special infinite series, 
 \emph{Mathematica Slovaka } \textbf{56} (2006),   no.~4, 409--414.

\bibitem{Hancl_Tijdeman2004} J.~Han\v{c}l and R.~Tijdeman, On the irrationality of Cantor series, 
 \emph{J. Reine Angew. Math. } \textbf{571} (2004), 145--158. 

\bibitem{Kuhapatanakul_Laohakosol2001} P.~Kuhapatanakul and ~V.~Laohakosol,   Irrationality of some series with rational terms,
\emph{Kasetsart~J. (Nat. Sci.)} \textbf{35} (2001), 205--209.

\bibitem{Oppenheim1954} A.~Oppenheim,  Criteria for irrationality of certain classes of numbers, 
 \emph{Amer. Math. Monthly } \textbf{61} (1954), no.~4, 235--241.

\bibitem{S.Serbenyuk}  {S. Serbenyuk } On some generalizations of real numbers representations, arXiv:1602.07929v1 (in Ukrainian)

\bibitem{Ser2017}
S.~ Serbenyuk. Rational numbers in terms of positive Cantor series,  \emph{Bull. Taras Shevchenko Natl. Univ. Kyiv Math. Mech.} \textbf{36} (2017), no.~2,  11--15 (in Ukrainian)

\bibitem{S13}  S. Serbenyuk.  Cantor series and rational numbers, available at  https://arxiv.org/pdf/1702.00471.pdf

\bibitem{Serbenyuk17}  S. Serbenyuk. Cantor series expansions of rational numbers, arXiv:1706.03124 or available at https://www.researchgate.net/publication/317099134

\bibitem{Serbenyuk2017}
Symon Serbenyuk. Representation of real numbers by the alternating Cantor series, \emph{ Integers} {\bf  17}(2017), Paper No. A15, 27 pp.

\bibitem{Serbenyuk20}  S. Serbenyuk, A note on expansions of rational numbers by certain series, \emph{Tatra Mountains Mathematical Publications} \textbf{77} (2020), 53--58, DOI: 10.2478\textfractionsolidus tmmp-2020-0032,  arXiv:1904.07264

\bibitem{Tijdeman_Pingzhi2002} Robert Tijdeman and Pingzhi Yuan, On the rationality of Cantor and Ahmes series, 
\emph{Indag. Math. (N.S.)}   \textbf{13}  (2002), no.~3, 407--418.





\end{thebibliography}
\end{document}